\documentclass[11pt]{amsart}
\usepackage{amsmath}
\usepackage{amssymb,amscd,latexsym,paralist}
\usepackage{mathbbol}
\usepackage{stmaryrd}
\usepackage{accents}
\usepackage{bm}
\usepackage[all]{xy}

\newcommand{\A}{{\mathbb{A}}}

\newcommand{\G}{\mathbb{G}}
\newcommand{\N}{{\mathbb{N}}}

\newcommand{\Q}{{\mathbb{Q}}}

\newcommand{\uS}{\underline{S}}

\newcommand{\Z}{{\mathbb{Z}}}

\newcommand{\bfa}{\underline{a}}
\newcommand{\bfb}{\underline{b}}
\newcommand{\bfc}{\underline{c}}

\newcommand{\const}{\mathrm{const}}

\newcommand{\id}{\mathrm{id}}
\newcommand{\ind}{\mathrm{ind}}

\renewcommand{\mod}{\;\mathrm{mod}\;}

\newcommand{\pr}{\mathrm{pr}}
\newcommand{\proj}{\mathrm{proj}}

\newcommand{\res}{\mathrm{res}}

\newcommand{\spec}{\mathrm{spec}\,}
\newcommand{\triv}{\mathrm{triv}}
\newcommand{\Aut}{\mathrm{Aut}}
\newcommand{\Hom}{\mathrm{Hom}}

\newcommand{\Iso}{\mathrm{Iso}}

\newcommand{\oW}{\overline{W}}
\newcommand{\ubA}{\underaccent{\tilde}{A}}
\newcommand{\ubB}{\underaccent{\tilde}{B}}
\newcommand{\ubGh}{\underaccent{\tilde}{\Gh}}
\newcommand{\ubM}{\underaccent{\tilde}{M}}
\newcommand{\ubW}{\underaccent{\tilde}{W}}

\newcommand{\tA}{\tilde{A}}

\newcommand{\Ah}{{\mathcal A}}

\newcommand{\Ch}{{\mathcal C}}

\newcommand{\Fh}{{\mathcal F}}
\newcommand{\Gh}{{\mathcal G}}

\newcommand{\Oh}{{\mathcal O}}

\newcommand{\Rh}{{\mathcal R}}

\newcommand{\Vh}{{\mathcal V}}

\newcommand{\propdiv}{\underset{\neq}{|}}
\newcommand{\silo}{\xrightarrow{\sim}}

\newcommand{\hullet}{\raisebox{0.05cm}{$\,\scriptscriptstyle \bullet\,$}}
\newcommand{\verk}{\mbox{\scriptsize $\,\circ\,$}}
\newcommand{\peq}{\preccurlyeq}

\newtheorem{theorem}{Theorem}
\newtheorem{lemma}[theorem]{Lemma}
\newtheorem{prop}[theorem]{Proposition}
\newtheorem{cor}[theorem]{Corollary}
\newenvironment{example}{\noindent {\bf Example}}{}

\newtheorem{exmp}[theorem]{Example}

\newenvironment{claimnn}{\noindent {\bf Claim}}{}

\newenvironment{defn}{\noindent {\bf Definition}\it}{}
\newenvironment{rem}{\noindent {\bf Remark}}{}

\begin{document}
\title{The universal deformation of the Witt ring scheme}
\author{Christopher Deninger \and Young-Tak Oh}
\date{\today}
\maketitle
\section{Introduction}

Witt vectors play an important role in several branches of mathematics. Combinatorial considerations led to the study of certain $q$-deformations over $\spec \Z [q]$ of the big Witt vector scheme $W$ over $\spec \Z$, c.f. \cite{O2}. In the present note we consider another and simpler $q$-deformation of $W$ to a non-unital ring scheme $W^{(q)}$. The main result asserts that $W^{(q)}$ enhanced by a Frobenius lift, Verschiebung and the choice of a coordinate for the first component is the universal deformation over reduced bases of $W$ with the corresponding structures. It follows that the triple ($W$, Frobenius lift, Verschiebung) has no deformations within unital ring-schemes and a one-parameter deformation in the non-unital category with parameter ``space'' $\G_a / \G_m$ over $\Z [q]$. The $q$-deformation of \cite{O2} turns out to be isomorphic to $W^{(q)}$.

Over non-reduced bases  the deformation theory of the triple $(W$, Frobenius lift, Verschiebung) is richer and Theorem \ref{t3} in section \ref{sec:2} allows to determine it in principle. The theorem also allows a simple proof of the decomposition theorem $W_{S \cdot T} = W_S \verk W_T$ for coprime divisor stable sets $S$ and $T$ due to Auer, c.f. \cite{A}. At the end of section \ref{sec:5} we discuss the deformations of $W$ for integer values of $q$ studied in \cite{L} and \cite{O1}.

Throughout the paper we work with Witt vectors for divisor stable sets. Since in general Frobenius and Verschiebung are not endomorphisms of those, one needs to work with projective systems of rings indexed by divisor stable subsets. As our main technical tool we prove a Cartier--Dieudonn\'e theorem for them. A direct proof is given in section \ref{sec:1}. A more conceptual proof is also possible using the theory of Witt vectors for certain inductive systems of rings which is developed in the appendix. This theory is a very natural generalization of Witt vector theory for individual rings and may be of independent interest. 
\section{$q$-deformed Witt vectors} \label{sec:2neu}
Convention: In this note all rings and algebras will be commutative and associative but not always unital. As usual ``non-unital'' means ``not necessarily unital''.

A non-empty subset $S$ of the natural numbers $\N = \{ 1,2, \ldots \}$ is called {\it divisor stable} if $n \in S$ and $d \mid n$ imply that $d \in S$. In particular $1 \in S$. For $n \in S$ the sets $S / n = \{ \nu \in S \mid \nu n \in S \} \subset S$ and $S (n) = \{ \nu \in S \mid n \nmid \nu \} \subset S$ are again divisor stable. We assume that the reader is familiar with the rings of $S$-Witt vectors $W_S (A)$ defined for all commutative rings $A$. The ring $W_S (A)$ is unital if and only if $A$ is unital. Set theoretically $W_S (A) = A^S$ and hence $W_S$ is a (unital) ring-scheme over $\spec \Z$ whose underlying scheme is $\A^S = \spec \Z [S]$. Here $\Z [S] = \Z [t_n \mid n \in S]$ is the free commutative unital algebra on the set $S$ or in other words the polynomial algebra in indeterminates $t_n$ for $n \in S$. The values of co-addition and co-multiplication
\[
\Delta_+ , \Delta_{\hullet} : \Z [S] \longrightarrow \Z [S] \otimes_{\Z} \Z [S] = \Z [x_n ; y_m \mid n,m \in S ]
\]
on the generators $n \in S$ are the Witt polynomials $\Sigma_n$ and $\Pi_n$ defining addition and multiplication in $W_S (A) = A^S$ i.e.:
\[
\Sigma_n = \Delta_+ (t_n) , \Pi_n = \Delta_{\hullet} (t_n) \quad \text{in} \; \Z [x_{\nu} ; y_{\mu} \mid \nu | n , \mu | n] \; .
\]
Here $x_{\nu} = t_{\nu} \otimes 1$ and $y_{\mu} = 1 \otimes t_{\mu}$. For every divisor stable subset $S'$ of $S$ there is a natural projection morphism $\pi : W_S \to W_{S'}$ of ring-schemes. There is a morphism $s : W_{S'} \to W_S$ of schemes such that $\pi \verk s = \id$. Equivalently $\pi : W_S (A) \to W_{S'} (A)$ is surjective for all rings $A$. For $n \in S$ Verschiebung $V_n : W_{S/n} \to W_S$ is an additive morphism which fits into an exact sequence
\[
0 \longrightarrow W_{S/n} \xrightarrow{V_n} W_S \xrightarrow{\pi} W_{S(n)} \longrightarrow 0 \; .
\]
Since $W_{\{ 1 \} } = \G_a$ additively, it follows that for finite $S$ the underlying additive group scheme of $W_S$ is an iterated extension of $\G_a$'s. In particular $W_S \otimes \Q$ is a commutative unipotent group scheme over $\spec \Q$ and hence the logarithm provides an isomorphism
\[
\log : W_S \otimes \Q \silo T_0 W_S \otimes \Q
\]
of additive group schemes over $\spec \Q$ c.f. \cite{DG} IV, \S\,2 n$^{\text o}$ 4 Proposition 4.1 (iii). Here $T_0 W_S$ is the tangent group scheme to $W_S$ along the zero section. It is canonically identified with $\A^S$ with componentwise addition. Consider the morphism $\uS : \A^S \to \A^S$ which on $A$-valued points maps $(x_d)_{d \in S}$ to $(d x_d)_{d \in S}$. Then the additive isomorphism
\[
\log_S = \uS \verk \log : W_S \otimes \Q \silo T_0 W_S \otimes \Q
\]
is induced by a unique additive morphism
\[
\log_S : W_S \longrightarrow T_0 W_S \cong \A^S \; .
\]
This is the so-called ghost map. Explicitly, we have
\[
\log_S ((a_d)_{d \in S}) = \Big( \sum_{d \mid n} d a^{n/d}_d \Big)_{n \in S}
\]
on $A$-valued points. The preceding assertions also hold for arbitrary $S$ as one sees by taking the projective limit over the finite divisor-stable subsets $S_0$ of $S$. In the standard approach to Witt vector theory, the ghost map is used to define the ring-structure. A new approach giving the ring-structure directly was introduced in \cite{CD}. The morphism $\log_S$ is a morphism of (unital) ring-schemes if the ring-scheme structure on $T_0 W_S \cong \A^S$ is defined componentwise and $\A^1$ carries the standard ring-structure. With the identifications $\Gamma (W_S, \Oh) = \Z [x_d \mid d \in S]$ and $\Gamma (T_0 W_S, \Oh) = \Z [u_d \mid d \in S]$ the induced ring-homomorphism
\[
\log^*_S : \Gamma (T_0 W_S, \Oh) \longrightarrow \Gamma (W_S , \Oh)
\]
is determined by the formulas
\[
\log^*_S (u_n) = \sum_{d \mid n} d x^{n/d}_d \quad \text{for} \; n \in S \; .
\]
These are also the components of the formal logarithm with Jacobian $\uS$ for the formal $|S|$-dimensional group-law determined by the polynomials $\Sigma_n$ for $n \in S$. 

{\bf Remark.} The map $\uS : A^S \to A^S$ on the product ring $A^S$ is additive. Viewing the elements of $A^S$ as formal power series in $x$ over $A$ with exponents in $S$ we have $\uS = x \frac{d}{dx}$. In this context recall that if additively $W (A)$ is viewed as the formal multiplicative group of power series over $A$ the ghost map is given by $x \frac{d}{dx} \verk \log$.  

We now recall the Frobenius and Verschiebung morphisms. For $n \in S$, the Frobenius morphism $F_n : W_S \to W_{S/n}$ is a morphism of ring-schemes. The following relations hold in an evident sense\\
$V_n \verk V_m = V_{nm} : W_{S/nm} \to W_S \; , \, F_n \verk F_m = F_{nm} : W_S \to W_{S/nm}$ if $nm \in S$\\
$F_n \verk V_n = n \, \id : W_{S/n} \to W_{S/n}$ for $n \in S$\\
$F_n \verk V_m = V_m \verk F_n : W_{S/m} \to W_{S/n}$ for $(n,m) = 1$ and $nm \in S$.\\
For a prime number $p \in S$, the Frobenius morphism $F_p$ reduces $\mod p$ to the $p$-th power map: the following diagram commutes for all rings $A$
\begin{equation}
\label{eq:1a}
\xymatrix{
W_S (A) \ar[rr]^{F_p} \ar[d] & & W_{S/p} (A) \ar[d] \\
W_S (A) / p \ar[r]^{(\,)^p} & W_S (A) / p \ar[r]^-{\pi} & W_{S/p} (A) / p \; .
}
\end{equation}
Here the vertical maps are the reduction maps $\mod p$. Of course the bottom line could also be replaced by
\[
W_S (A) / p \xrightarrow{\pi} W_{S/p} (A) / p \xrightarrow{(\,)^p} W_{S/p} (A) / p \; .
\]

Finally consider the functor $P$ on rings defined by $P (A) = (A , \cdot)$. It is represented by the monoid scheme $P = \spec \Z [t]$ with co-multiplication $\Delta_{\hullet}$ and co-unit $\varepsilon_1$ determined by $\Delta_{\hullet} (t) = t \otimes t$ and $\varepsilon_1 (t) = 1$. The Teichm\"uller map is the multiplicative morphism
\[
\omega : P \longrightarrow W_S
\]
which on $A$-valued points sends $a \in A = P (A)$ to $(a \delta_{d,1})_{d \in S}$ in $A^S = W_S (A)$. Here $\delta_{\nu , \mu} = 1$ if $\nu = \mu$ and $\delta_{\nu, \mu} = 0$ if $\nu \neq \mu$. For every commutative unital ring $R$, the whole situation can be base-changed to $\spec R$ and we then speak of Witt vector schemes over $R$ etc.

We now describe two deformations of Witt vector theory which will later turn out to be isomorphic and universal in a suitable sense. The first one, $W^{(q)}$ is obtained by a simple modification of the usual Witt vector functor. The second one, $\oW^{1-q}$ was introduced in \cite{O2} using a $q$-deformed ghost map.

Let $A$ be a unital algebra over the polynomial ring $\Z [q]$. We denote by $A^{(q)}$ the ring with underlying additive group $(A , +)$ and twisted multiplication $x \ast y = q x y$. It is unital if and only if $A$ is a $\Z [q, q^{-1}]$-algebra. Setting $W^{(q)}_S (A) = W_S (A^{(q)})$ we a obtain a commutative non-unital ring scheme $W^{(q)}_S$ over $\spec \Z [q]$ whose underlying scheme is $\A^S$. The base change of $W^{(q)}_S$ to $\spec \Z [q , q^{-1}]$ is a unital ring-scheme. By the Yoneda lemma the Frobenius and Verschiebung maps for $n \in S$
\[
F_n : W_S (A^{(q)}) \longrightarrow W_{S/n} (A^{(q)}) \quad \text{and} \quad V_n : W_{S/n} (A^{(q)}) \longrightarrow W_S (A^{(q)})
\]
come from morphisms $F_n : W^{(q)}_S \to W^{(q)}_{S/n}$ and $V_n : W^{(q)}_{S/n} \to W^{(q)}_S$. They have the same properties as the ones recalled above for the usual Witt vector schemes. In particular the commutative group schemes $W^{(q)}_S$ are unipotent over $\spec \Q [q]$ and hence there is the $\log$-isomorphism
\[
\log : W^{(q)}_S \otimes \Q \silo T_0 W^{(q)}_S \otimes \Q
\]
of additive group schemes over $\spec \Q [q]$. Here $T_0 W^{(q)}_S$ is the tangent group scheme to $W^{(q)}_S$ over $\spec \Z [q]$ along the zero section. As before $T_0 W^{(q)}_S \cong \A^S$ canonically and we have the morphism $\uS$ defined as before. The additive isomorphism
\[
\log_S = \uS \verk \log : W^{(q)}_S \otimes \Q \silo T_0 W^{(q)}_S \otimes \Q
\]
is induced by a unique additive morphism over $\Z [q]$
\[
\log_S : W^{(q)}_S \longrightarrow T_0 W^{(q)}_S \cong \A^S \; .
\]
On $A$-valued points the induced (ghost-)map
\begin{equation}
\label{eq:1n}
\log_S : A^S \longrightarrow (A^{(q)})^S
\end{equation}
is given by the formula
\[
\log_S ((a_d)_{d \in S}) = \big( \sum_{d \mid n} d q^{\frac{n}{d}-1} a^{\frac{n}{d}}_d \Big)_{n \in S} \; .
\]
Note here that $a^{\ast n} = q^{n-1} a^n$ for $a \in A , n \ge 1$ by definition of the $q$-twisted multiplication in $A^{(q)}$. The map $\log_S$ in \eqref{eq:1n} is a ring-homomorphism if on the right we view $(A^{(q)})^S$ as a ring under componentwise addition and multiplication. On the left we take the ring-structure on the set $A^S$ coming from the identification $W^{(q)}_S (A) = A^S$ as sets. Via the identifications $\Gamma (W^{(q)}_S , \Oh) = \Z [q] [x_d \mid d \in S]$ and $\Gamma (T_0 W^{(q)}_S , \Oh) = \Z [q] [u_d \mid d \in S]$ the $\Z [q]$-algebra homomorphism
\[
\log^*_S : \Gamma (T_0 W^{(q)}_S , \Oh) \longrightarrow \Gamma (W^{(q)}_S , \Oh)
\]
is given by the formulas
\[
\log^*_S (u_n) = \sum_{d \mid n} d q^{\frac{n}{d}-1} x^{d/n}_d = q^{-1} \sum_{d \mid n} d (q x_d)^{n/d} \; .
\]
It follows that the universal polynomials for addition and multiplication and the Frobenius and Verschiebung morphisms are obtained from the usual ones by multiplying the variables by $q$ and dividing the resulting polynomial by $q$. For $S = \{ 1,p \}$ for example, setting $a = (a_1 , a_p) , b = (b_1 , b_p)$ we have:
\begin{align*}
\Sigma_1 (a , b) & = a_1 + b_1 \\
\Sigma_p (a,b) & = a_p + b_p - q^{p-1} \sum^{p-1}_{\nu = 1} p^{-1} {p \choose \nu} a^{\nu}_1 b^{p-\nu}_1 \\
\Pi_1 (a , b) & = q a_1 b_1 \\
\Pi_p (a,b) & = q p a_p b_p + q^p (a^p_1 b_p + a_p b^p_1) \; .
\end{align*}
Consider the functor $P^{(q)}$ on $\Z [q]$-algebras defined by $P^{(q)} (A) = (A^{(q)} , \ast) = (A , \ast)$ where $x \ast y = q x y$ as above. It is represented by the semigroup scheme $P^{(q)} = \spec \Z [q] [t]$ over $\spec \Z [q]$ with co-multiplication $\Delta_{\hullet}$ determined by $\Delta_{\hullet} (t) = qt \otimes t$. The base change to $\Z [q , q^{-1}]$ is a monoid scheme with co-unit $\varepsilon_1 (t) = q^{-1}$. There is a unique morphism $\omega^{(q)} : P^{(q)} \to W^{(q)}_S$ of multiplicative semigroup schemes which on $\Z [q]$-algebras $A$ becomes the ordinary Teichm\"uller map
\[
(A^{(q)} , \hullet) \longrightarrow W_S (A^{(q)}) \; , \; \omega^{(q)} (a) = (a \delta_{d,1})_{d \in S} \; .
\]
Base changed to $\Z [q, q^{-1}]$ the map $\omega^{(q)}$ becomes a morphism of monoid schemes.

\begin{rem}
In \cite{L} and \cite{O1} certain $q$-deformations of $W$ over $\spec \Z$ were studied for integer values of $q$. The ghost map is the same as above but the induced multiplication on the Witt vectors is different since the ghost side is viewed as the ring $A^S$ and not as $(A^{(q)})^S$. The additive structure is the same though.
\end{rem}

For every polynomial $g (q)$ in $\Z [q]$ a ring scheme $\oW^{g (q)}$ over $\Z [q]$ was introduced in \cite{O2}. The construction can be generalized to every divisor-stable subset $S$ of $\N$ using the same arguments. The most relevant case for us is $\oW^{1-q}_{\!\!S}$. It is defined as follows: There is a unique functorial ring structure on $A^S$ for all $\Z [q]$-algebras $A$ such that the following ``ghost'' map is a homomorphism of non-unital rings
\[
\Gh_S : A^S \longrightarrow (A^{(q)})^S \; , \; \Gh_S ((a_n)_{n \in S}) = \Big( \sum_{d \mid n} dq^{-1} (1 - (1-q)^{n/d}) a^{n/d}_d \Big)_{n \in S} \; .
\]

It follows from Propositions 6.1 and 6.2 of \cite{O2} that there are unique Frobenius and Verschiebung morphisms $F_n : \oW^{1-q}_{\!\!S} \to \oW^{1-q}_{S/n}$ and $V_n : \oW^{1-q}_{S/n} \to \oW^{1-q}_{\!\!S}$ that correspond via the ghost maps to the maps
\[
F_n ((a_{\nu})_{\nu \in S}) = (a_{n\nu})_{\nu \in S/n} \quad \text{and} \quad V_n ((a_{\nu})_{\nu \in S/n}) = (n \delta_{n\mid \nu} a_{\nu/n})_{\nu \in S} \; .
\]
They satisfy the same relations as for the usual Witt vectors and in particular $F_p$ reduces $\mod p$ to the $p$-th power morphism. The modified $\log$-morphism $\log_S = \uS \verk \log$ exists over $\Z [q]$, i.e. $\log_S : \oW^{1-q}_{\!\!S} \to T_0 \oW^{1-q}_{\!\!S}$ and on $A$-valued points it equals $\Gh_S$. Writing 
\[
\Gamma (\oW^{1-q}_{\!\!S} , \Oh) = \Z [q] [x_d \mid d \in S]\quad \text{and} \quad \Gamma (T_0 \oW^{1-q}_{\!\!S} , \Oh) = \Z [q] [u_d \mid d \in S]
\]
we have
\[
\log^*_S (u_n) = \sum_{d\mid n} dq^{-1} (1 - (1-q)^{n/d}) x^{n/d}_d \quad \text{for} \; n \in S \; .
\]
The universal polynomials describing $\oW^{1-q}_{\!\!S}$ are more complicated than the ones for $W^{(q)}_S$. For $S = \{ 1, p \}$ for example they are the following:
\begin{align*}
\Sigma_1 (a , b) & = a_1 + b_1 \\
\Sigma_p (a,b) & = a_p + b_p - h (q) \sum^{p-1}_{\nu-1} p^{-1} {p \choose \nu} a^{\nu}_1 b^{p-\nu}_1 \\
\Pi_1 (a , b) & = q a_1 b_1 \\
\Pi_p (a,b) & = qpa_p b_p + qh (q) (a^p_1 b_p + b^p_1 a_p) + qh (q) r (q) a^p_1 b^p_1 \; .
\end{align*}
Here we have set
\[
h (q) = q^{-1} (1 - (1 -  q)^p) = 1 + (1-q) + \ldots + (1-q)^{p-1}
\]
and
\[
r (q) = p^{-1} (h (q) - q^{p-1}) = p^{-1} q^{-1} (1 - q^p - (1-q)^p) \in \Z [q] \; .
\]
For the integrality of $r (q)$ note that
\[
1 - q^p - (1 - q)^p \equiv 1 - q^p - (1 - q^p) \mod p \equiv 0 \mod p
\]
and
\[
1 - q^p - (1-q)^p \equiv 0 \mod q \; .
\]
For any polynomial $g (q) \in \Z [q]$ consider the ring homomorphism $\alpha : \Z [q] \to \Z [q]$ with $\alpha (q) = 1 -g (q)$ and set
\[
\oW^{g (q)}_{\!\!S} = \oW^{1-q}_{\!\!S} \otimes_{\Z [q] , \alpha} \Z [q] \; .
\]
In \cite{O2} the non-unital ring schemes $\oW^{g (q)}_{\!\!S}$ and their underlying group schemes were investigated in some detail for $S = \N$. It is possible to prove directly that $W^{(q)}_S$ and $\oW^{1-q}_{\!\!S}$ together with their extra structures are isomorphic. However this also follows without effort from the universality property of $W^{(q)}_S$, c.f. Corollary \ref{t7} and the subsequent example. 
\section{A variant of the Cartier--Dieudonn\'e lemma} \label{sec:1}
In this section we prove two technical results which are the basis for the deformation theory of the Witt vector scheme in the next section. For a divisor stable subset $T$ of $\N$ we write $S \peq T$ to signify that $S$ is a divisor stable subset of $T$. Let $\ubA = (A_S)_{S \peq T}$ be a projective system of rings on $T$ i.e. a contravariant functor from the ordered set $\{ S \peq T \}$ viewed as a category to the category of (commutative) unital or non-unital rings. For $S_1 \peq S_2 \peq T$ we denote the transition maps simply by $\pi : A_{S_2} \to A_{S_1}$. We say that $\ubA$ is equipped with commuting Frobenius lifts if for all prime numbers $p \in S \peq T$ there are ring homomorphisms
\[
F_p : A_S \longrightarrow A_{S/p}
\]
with the following properties:\\
1) For all $a \in A_S$ we have the congruence
\[
F_p (a) \equiv \pi (a)^p \mod p A_{S/p} \; .
\]
2) The $F_p$ are natural in the sense that for $p \in S_1 \peq S_2 \peq T$ the diagram
\[
\xymatrix{
A_{S_2} \ar[r]^{F_p} \ar[d]_{\pi} & A_{S_2 /p} \ar[d]^{\pi} \\
A_{S_1} \ar[r]^{F_p} & A_{S_1/p}
}
\]
commutes.\\
3) For prime numbers $l$ with $pl \in S \peq T$, the diagram
\[
\xymatrix{
A_S \ar[r]^{F_p} \ar[d]_{F_l} & A_{S/p} \ar[d]^{F_l} \\
A_{S/l} \ar[r]^{F_p} & A_{S/pl}
}
\]
commutes.

For each $n \in S$ we define $F_n : A_S \to A_{S/n}$ as the composition $F_n = F^{\nu_1}_{p_1} \verk \ldots \verk F^{\nu_r}_{p_r}$ where $n = p^{\nu_1}_1 \cdots p^{\nu_r}_r$ is the prime decomposition of $n$. By 2) this is well defined and the naturality and commutation properties 2) and 3) then hold without the assumption that $p$ and $l$ are prime numbers. Morphisms of projective systems of rings on $T$ with commuting Frobenius lifts are defined in the obvious way and we obtain a category $\Rh \Fh_T$ both in the unital and in the non-unital cases. Note that for a ring $A$, Witt vector theory gives us an object $\ubW (A) := (W_S (A))_{S \peq T}$ of $\Rh \Fh_T$. Generalizing a well known fact from the theory of Witt vector rings we show that $\ubW$ has the following universal property:

\begin{prop}
\label{t1}
Assume that for $\ubA$ in $\Rh \Fh_T$ the ring $A = A_{\{ 1 \} }$ has no $T$-torsion. Then there is a unique morphism $\alpha = (\alpha_S)_{S \peq T} : \ubA \to \ubW (A)$ in $\Rh \Fh_T$ with $\alpha_S = \id_A$ for $S = \{ 1 \}$. The morphism $\alpha$ is functorial in $\ubA$. Explicitly it is given as follows. The composition of $\alpha_S : A_S \to W_S (A)$ with the (injective) ghost map $\log_S : W_S (A) \to A^S$ is given by the formula
\[
(\log_S \verk \alpha_S) (a) = (\pi F_n (a))_{n \in S} \quad \text{for} \; a \in A_S \; .
\]
Here $\pi$ denotes the map $\pi : A_{S/n} \to A_{ \{ 1 \} } = A$.
\end{prop}

\begin{proof}
We have to show that there is a unique family of ring homomorphisms $\alpha_S : A_S \to W_S (A)$ for $S \peq T$ which commute with the projections $\pi$ and the Frobenius maps. Since $A$ has no $T$-torsion by assumption, the ghost maps
\[
\log_S : W_S (A) \longrightarrow A^S
\]
give isomorphisms onto their images $X_S (A) := \log_S (W_S (A))$. On the ghost side the projection $\pi : X_{S_2} (A) \to X_{S_1} (A)$ for $S_1 \peq S_2$ is induced by the projection $A^{S_2} \to A^{S_1}$. The Frobenius map
\[
F_n : X_S (A) \longrightarrow X_{S/n} (A) \quad \text{for} \; n \in S \peq T
\]
is the restriction of the map
\[
F_n : A^S \longrightarrow A^{S/n} \; , \; (a_{\nu})_{\nu \in S} \longmapsto (a_{nd})_{d \in S/n} \; .
\]

The uniqueness assertion follows from the next claim:

\begin{claimnn}
There is a unique family of ring homomorphisms 
\[
\beta_S : A_S \longrightarrow A^S \quad \text{for} \; S \peq T \; , \; \text{where} \; \beta_{\{ 1 \} } = \id
\]
commuting with the respective projections $\pi$ and Frobenius maps $F_p$. It is given by the formula
\begin{equation}
\label{eq:1}
\beta_S (a) = (\pi F_n (a))_{n \in S} \quad \text{for} \; a \in A_S
\end{equation}
where $\pi : A_{S/n} \to A_{\{ 1 \} } = A$.
\end{claimnn}

As for the uniqueness and formula \eqref{eq:1} we argue as follows. Write $\beta_S (a) = (x_n)_{n \in S}$. For $n \in S$ set $S' = \{ d \in S \mid d \mid n \}$. Then $S' \peq S$ and $S'/n = \{ 1 \}$. Consider the commutative diagram:
\[
\xymatrix{
A_S \ar[r]^{\pi} \ar[d]_{\beta_S} & A_{S'} \ar[r]^{F_n} \ar[d]_{\beta_{S'}} & A \ar[d]^{\beta_{\{ 1 \} } = \id } \\
A^S \ar[r]^{\pi} & A^{S'} \ar[r]^{F_n} & A \; .
}
\]
We find
\begin{align*}
x_n & = F_n ((x_d)_{d \in S'}) = F_n \pi \beta_S (a) = F_n \pi (a) \\
& = \pi F_n (a) \; .
\end{align*}
Here the final equality is due to the commutative diagram where $A_{S'/n} = A$
\[
\xymatrix{
A_S \ar[r]^{\pi} \ar[d]_{F_n} & A_{S'} \ar[d]^{F_n} \\
A_{S/n} \ar[r]^-{\pi} & A_{S'/n} \; .
}
\]
Thus $\beta = (\beta_S)$ must be given by \eqref{eq:1}. It is straightforward to check that \eqref{eq:1} does indeed define the desired family of maps.

For the proof of the proposition it remains to show that $\beta_S (A_S) \subset X_S (A)$ for all $S \peq T$. Thus for every $a \in A_S$ there have to be (uniquely determined) elements $a_d \in A$ for $d \in S$ with
\begin{equation}
\label{eq:2}
\pi F_n (a) = \sum_{d \mid n} d a^{n/d}_d \quad \text{in $A$ for every} \; n \in S \; .
\end{equation}
For this we prove the following stronger statement by induction on $n \in S$.

\begin{claimnn}
Given $n \in S$ and $b \in A_S$, there exist elements $b_d \in A_{S/d}$ for all $d \mid n$ such that we have
\begin{equation}
\label{eq:3}
F_n (b) = \sum_{d \mid n} d \pi (b_d)^{n/d} \quad \text{in} \; A_{S/n} \; .
\end{equation}
Here $\pi (b_d)$ is the projection of $b_d$ along $\pi : A_{S/d} \to A_{S/n}$. 
\end{claimnn}

For $n = 1$ we take $b_1 = b$. Assume that the claim holds for proper divisors of $n$. For any prime divisor $p$ of $n$ we then know that
\[
F_{n/p} (b) = \sum_{d \mid (n/p)} d \pi (b_d)^{n / pd} \quad \text{in} \; A_{S / (n/p)}
\]
for elements $b_d \in A_{S/d}$ and corresponding $\pi : A_{S/d} \to A_{S / (n/p)}$. Applying $F_p$ we get
\[
F_n (b) = \sum_{d \mid (n/p)} d \pi F_p (b_d)^{n /pd} \quad \text{in} \; A_{S/n}
\]
where the $\pi$'s are projections $A_{S/dp} \to A_{S/n}$. By assumption
\[
F_p (b_d) \equiv \pi (b_d)^p \mod p A_{S/dp}
\]
and therefore
\[
F_p (b_d)^{n/pd} \equiv \pi (b_d)^{n/d} \mod p^{v_p (n/d)} A_{S/dp} \; .
\]
Here we have used the fact that $\alpha \equiv \beta \mod p$ implies that $\alpha^{p^i} \equiv \beta^{p^i} \mod p^{i+1}$ and hence $\alpha^k \equiv \beta^k \mod p^{v_p (k) + 1}$ for $k \ge 1$. It follows that
\begin{align*}
F_n (b) & \equiv \sum_{d \mid (n/p)} d \pi (b_d)^{n/d} \mod p^{v_p (n)} A_{S/n}  \\
& \equiv \sum_{d \propdiv n} d \pi (b_d)^{n/d} \mod p^{v_p (n)} A_{S/n} \; .
\end{align*}
Here and in the following the notation $d \propdiv n$ means that $d \mid n$ and $d \neq n$.

For the last step, note that if $d \mid n$ and $d \nmid (n/p)$ then $v_p (d) = v_p (n)$. Since $p \mid n$ was arbitrary we conclude that
\[
F_n (b) \equiv \sum_{d \underset{\neq}{|} n} d \pi (b_d)^{n/d} \mod n A_{S/n} \; .
\]
Hence an element $b_n \in A_{S/n}$ can be found so that \eqref{eq:3} holds. The explicit formula for $\log_S \verk \alpha_S$ shows that $\alpha = (\alpha_S)$ depends functorially on $\ubA$. 
\end{proof}

\begin{rem}
In the appendix we sketch a theory of Witt vector rings for ind-rings which elucidates the somewhat ad hoc proof of Proposition \ref{t1}.
\end{rem}

In general the morphism $\alpha$ in Proposition \ref{t1} will not be an isomorphism. For this more structure is required. We call a projective system $\ubA = (A_S)_{S \peq T}$ {\it continuous} if for all $S \peq T$ we have $A_S = \varprojlim_{S_0 \peq S} A_{S_0}$ where $S_0$ runs over the {\it finite} divisor stable subsets of $S$. Note that for finite $T$ continuity is automatic. Consider the category $\Rh\Fh\Vh_T$ whose objects are continuous projective systems $\ubA = (A_S)_{S \peq T}$ of rings on $T$ with commuting Frobenius lifts $F_p$ together with Verschiebung maps for all prime numbers $p \in S \peq T$
\[
V_p : A_{S/p} \longrightarrow A_S 
\]
with the following properties:\\
4) The $V_p$ are additive homomorphisms which are natural in the sense that for $p \in S_1 \peq S_2 \peq T$ the diagram
\[
\xymatrix{
A_{S_2/p} \ar[r]^{V_p} \ar[d]_{\pi} & A_{S_2} \ar[d]^{\pi} \\
A_{S_1/p} \ar[r]^{V_p} & A_{S_1}
}
\]
commutes.\\
5) For prime numbers $l$ with $pl \in S \peq T$ the diagram
\[
\xymatrix{
A_{S/p l} \ar[r]^{V_p} \ar[d]_{V_l} & A_{S/l} \ar[d]^{V_l} \\
A_{S/p} \ar[r]^{V_p} & A_S
}
\]
commutes.\\
6) The composition $A_{S/p} \xrightarrow{V_p} A_S \xrightarrow{F_p} A_{S/p}$ is $p$-multiplication i.e. $F_p \verk V_p = p$ for $p \in S \peq T$. For any prime $l$ with $l \neq p$ and $pl \in S \peq T$ the diagram
\[
\xymatrix{
A_{S/p} \ar[r]^{V_p} \ar[d]_{F_l} & A_S \ar[d]^{F_l} \\
A_{S/pl} \ar[r]^{V_p} & A_{S/l}
}
\]
commutes. \\
7) There are exact sequences of additive groups
\[
0 \longrightarrow A_{S/p} \xrightarrow{V_p} A_S \xrightarrow{\pi} A_{S (p)} \longrightarrow 0 \; .
\]
Here $S (p) = \{ d \in S \mid p \nmid d \}$. 

Morphisms between objects of $\Rh\Fh\Vh_T$ are defined in the obvious way. For a ring $A$, Witt vector theory gives an object $\ubW (A) = (W_S (A))_{S \peq T}$ of $\Rh\Fh\Vh_T$.

\begin{prop}
\label{t2}
Assume that for $\ubA$ in $\Rh\Fh\Vh_T$ the ring $A = A_{\{ 1 \} }$ has no $T$-torsion. Then the morphism $\alpha : \ubA \to \ubW (A)$ in $\Rh \Fh_T$ from Proposition \ref{t1} defines an isomorphism in $\Rh\Fh\Vh_T$. In particular it is an isomorphism in $\Rh\Fh_T$.
\end{prop}

\begin{proof}
The main point is that for any prime $p \in S \peq T$ the diagram
\begin{equation}
\label{eq:4}
\xymatrix{
A_{S/p} \ar[r]^{V_p} \ar[d]_{\alpha_{S/p}} & A_S \ar[d]^{\alpha_S} \\
W_{S/p} (A) \ar[r]^{V_p} & W_S (A)
}
\end{equation}
commutes. Since $A$ has no $T$-torsion, the ghost map $\log_S$ on $W_S (A)$ is injective and hence it suffices to show that we have
\begin{equation}
\label{eq:5}
\log_S \verk V_p \verk \alpha_{S/p} = \log_S \verk \alpha_S \verk V_p \quad \text{on} \; A_{S /p} \; .
\end{equation}
For $a \in A_{S/p}$, using the explicit description of $\log_S \verk \alpha$ in Proposition \ref{t1} we find, setting $\delta_{p\mid n} = 1$ for $p \mid n$ and $= 0$ for $p \nmid n$:
\begin{align*}
(\log_S \verk \alpha_S \verk V_p) (a) & = (\log_S \verk \alpha_S) (V_p (a)) = (\pi F_n V_p (a))_{n \in S} \\
& = (\delta_{p \mid n} \pi F_n V_p (a))_{n \in S} \; .
\end{align*}
Namely, for $p \nmid n$ we have, by 6) and 7) above
\[
\pi F_n V_p = \pi V_p F_n = 0 \; .
\]
Using 6) again we therefore find
\begin{align*}
(\log_S \verk \alpha_S \verk V_p) (a) & = (\delta_{p \mid n} \pi F_{n/p} F_p V_p (a))_{n \in S} \\
& = p (\delta_{p \mid n} \pi F_{n/p} (a))_{n \in S} \\
& = V_p ((\pi F_n (a))_{n \in S/p}) \\
& = (V_p \verk \log_{S/p} \verk \alpha_{S/p}) (a) \\
& = (\log_S \verk V_p \verk \alpha_{S/p}) (a) \; .
\end{align*}
Thus we have shown \eqref{eq:5} and hence \eqref{eq:4}. Combining \eqref{eq:4} and 7) we get a commutative diagram with exact lines for all $p \in S \peq T$
\[
\xymatrix{
0 \ar[r] & A_{S/p} \ar[r]^{V_p} \ar[d]_{\alpha_{S/p}} & A_S \ar[r]^{\pi} \ar[d]_{\alpha_S}  & A_{S (p)} \ar[d]_{\alpha_{S (p)}} \ar[r] & 0 \\
0 \ar[r] & W_{S/p} (A) \ar[r]^{V_p} & W_S (A) \ar[r]^{\pi} & W_{S (p)} (A) \ar[r] & 0 
}
\]
Since $\alpha_{\{ 1 \} } = \id$, an induction with respect to $|S|$ now shows that $\alpha_S$ is an isomorphism for all finite $S$. The general case follows by the continuity of $\ubA$ and $\ubW$. 
\end{proof}
\section{Universality of Witt vector schemes} \label{sec:2}

As before, fix a divisor stable subset $T \subset \N$. Consider a projective system $\ubM = (M_S)_{S \peq T}$ of affine commutative ring-schemes $M_S = \spec B_S$ over a base ring $R$. As before, the transition maps are denoted by $\pi$. We say that $\ubM$ is equipped with commuting Frobenius lifts if for all prime numbers $p \in S \peq T$ there are morphisms of ring-schemes over $R$
\[
F_p : M_S \longrightarrow M_{S/p}
\]
such that for all $R$-algebras $C$, the projective system $\ubM (C) = (M_S (C))_{S \peq T}$ together with the maps $F_p (C)$ is an object of $\Rh\Fh_T$. In particular the relations $\pi \verk F_p = F_p \verk \pi$ and $F_p \verk F_l = F_l \verk F_p$ hold in the obvious sense, c.f. 2), 3) in section \ref{sec:1}. The ensuing category is denoted by $\Rh\Fh_{T/R}$. We may view its objects as the ``representable functors'' from the category of $R$-algebras to $\Rh\Fh_T$. We define another category $\Rh\Fh\Vh_{T/R}$ as follows. Objects are continuous projective systems $\ubM$ in $\Rh\Fh_{T/R}$ equipped with morphisms of the underlying additive group schemes $V_p : M_{S/p} \to M_S$ for all prime numbers $p \in S \peq T$ such that on $R$-algebras $C$ one gets objects of $\Rh\Fh\Vh_T$. Morphisms are the evident ones. In particular $\pi \verk V_p = V_p \verk \pi$ and $V_p \verk V_l = V_l \verk V_p$ and $F_p \verk V_p = p \, \id$, moreover $F_l \verk V_p = V_p \verk F_l$ for $p \neq l$ and the sequence
\begin{equation} \label{eq:7a}
0 \longrightarrow M_{S/p} \xrightarrow{V_p} M_S \xrightarrow{\pi} M_{S (p)} \longrightarrow 0
\end{equation}
is exact and $\pi$ is (scheme-theoretically) split. We may view the objects of $\Rh\Fh\Vh_{T/R}$ as the representable functors from the category of $R$-algebras to $\Rh\Fh\Vh_T$. 

Given any affine ring-scheme $M = \spec B$ over $R$, the functor 
\[
W^M_S = W_S \verk M : \{ R\text{-algebras} \} \longrightarrow \{ \text{rings} \}
\]
is an affine ring scheme over $R$ with underlying scheme $M^S = \spec B^{\otimes S}$. The Witt-vector Frobenius and Verschiebung maps for all prime numbers $p \in S$
\[
F_p : W_S (M (C)) \longrightarrow W_{S/p} (M (C)) \quad \text{and} \quad V_p : W_{S/p} (M (C)) \longrightarrow W_S (M (C))
\]
for the rings $M (C)$ where $C$ runs over $R$-algebras come by the Yoneda lemma from morphisms of ring schemes
\[
F_p : W^M_S \longrightarrow W^M_{S/p} \quad \text{and} \quad V_p : W^M_{S/p} \longrightarrow W^M_S \; .
\]
By the usual Witt-vector theory we see that equipped with the $F_p$'s (and $V_p$'s) the projective system of ring-schemes $\ubW^M := (W^M_S)_{S \peq T}$ becomes an object of $\Rh\Fh_{T/R}$ resp. $\Rh\Fh\Vh_{T/R}$.

We need a technical condition in the following:

\begin{defn}
An object $\ubM$ in $\Rh\Fh_{T/R}$ or $\Rh \Fh \Vh_{T/R}$ has no Hopf $T$-torsion if for $M = M_{\{ 1 \} }$ the abelian groups $M (R) , M (B_S)$ and $M (B_S \otimes_{\Z} B_S)$ have no $T$-torsion for all $S \peq T$ where $M_S = \spec B_S$. 
\end{defn}

\begin{theorem}
\label{t3}
a) Assume that $\ubM$ in $\Rh\Fh_{T/R}$ has no Hopf $T$-torsion. Then there is a unique morphism 
\[
\alpha = (\alpha_S)_{S \peq T} : \ubM \to \ubW^M
\]
in $\Rh\Fh_{T/R}$ with $\alpha_S = \id$ for $S = \{ 1 \}$. \\
b) Assume that $\ubM$ in $\Rh\Fh\Vh_{T/R}$ has no Hopf $T$-torsion. Then the above unique morphism $\alpha : \ubM \to \ubW^M$ in $\Rh\Fh_{T/R}$ defines an isomorphism in $\Rh\Fh\Vh_{T/R}$ (and hence in $\Rh\Fh_{T/R}$).
\end{theorem}

\begin{proof}
For the class $\Ch$ of $R$-algebras $C$ with $M (C)$ $T$-torsion free, Propositions \ref{t1} and \ref{t2} ensure the existence of unique morphisms in $\Rh\Fh_T$ (resp. isomorphisms in $\Rh\Fh\Vh_T$)
\[
\alpha (C) = (\alpha (C)_S)_{S \peq T} : \ubM (C) \longrightarrow \ubW^M (C) = \ubW (M (C)) \; .
\]
Moreover they are functorial in $C \in \Ch$. We only need to show that the functorial ring homomorphisms $\alpha (C)_S : M_S (C) \to W^M_S (C)$ come from uniquely determined morphisms of ring-schemes $\alpha_S : M_S \to W^M_S$ which commute with the maps $\pi , F_p , V_p$'s. Since $\ubM$ has no Hopf $T$-torsion, the existence of $\alpha_S$ follows from the next version of the Yoneda lemma applied to $F = M_S , G = W^M_S$ and the class $\Ch$ above. Another application of (the first part of) the lemma shows that the $\alpha_S$'s commute with $\pi , F_p , V_p$.
\end{proof}

\begin{lemma}
\label{t4}
Let $F = \spec A$ and $G = \spec B$ be two affine ring-schemes over $R$. Assume that for a class $\Ch$ of $R$-algebras there are functorial ring homomorphisms for all $C \in \Ch$,
\[
\alpha (C) : F (C) \longrightarrow G (C) \; .
\]
If $A \in \Ch$, then there is a unique morphism of $R$-schemes $\alpha : F \to G$ which induces $\alpha (C)$ for all $C \in \Ch$. If in addition $R \in \Ch$ and $A \otimes_R A \in \Ch$, then $\alpha$ is a morphism of ring-schemes.
\end{lemma}

\begin{proof}
The morphism $\spec \alpha^{\sharp} : F \to G$ induced by
\[
\alpha^{\sharp} = \alpha (A) (\id_A) \in \Hom_{R\text{-alg}} (B,A)
\]
induces the given maps $\alpha (C)$, since for $\psi \in F (C) = \Hom_{R\text{-alg}} (A,C)$ we have
\[
\alpha (C) (\psi) = \psi \verk \alpha^{\sharp} = (\spec \alpha^{\sharp}) (\psi) \; .
\]
On the other hand, this equation applied to $C = A , \psi = \id$ implies that $\alpha^{\sharp} = \alpha (A) (\id)$ is unique, hence $\alpha := \spec \alpha^{\sharp}$ is unique as well. Since $\alpha (C)$ is an additive (multiplicative) map for $C \in \Ch$ and since $\alpha (C) = (\alpha^{\sharp})^* , \alpha (C) \times \alpha (C) = (\alpha^{\sharp} \otimes \alpha^{\sharp})^*$, we have
\[
(\alpha^{\sharp})^* \verk \Delta^* = \Delta^* \verk (\alpha^{\sharp} \otimes \alpha^{\sharp})^*
\]
on $F (C) \times F (C) = \Hom_{R\text{-alg}} (A \otimes A, C)$ where $\Delta$ is the co-addition (co-multi\-pli\-ca\-tion) on $A$ resp. $B$. If $A \otimes A \in \Ch$, we can apply this to $C = A \otimes A$ and $\id \in \Hom (A \otimes A , A \otimes A)$ and obtain
\[
\Delta \verk \alpha^{\sharp} = (\alpha^{\sharp} \otimes \alpha^{\sharp}) \verk \Delta \; .
\]
If $R \in \Ch$, then $\alpha (R) = (\alpha^{\sharp})^*$ and from $\alpha (R) (0) = 0$ we get $(\alpha^{\sharp})^* (\varepsilon_A) = \varepsilon_B$ i.e. $\varepsilon_B = \varepsilon_A \verk \alpha^{\sharp}$ for the co-zeroes (or co-units) $\varepsilon_A$ and $\varepsilon_B$ of $A$ and $B$. 
\end{proof}

We end this section with an application of Theorem \ref{t3}. Two divisor stable subsets $T_1$ and $T_2$ of $\N$ are coprime if and only if $T_1 \cap T_2 = \{ 1 \}$. In this case the set $T_1 \cdot T_2 = \{ nm \mid n \in T_1 , m \in T_2 \}$ is again divisor stable. Lenstra conjectured and in \cite{A} Auer proved that there are natural isomorphisms of functors on unital rings
\begin{equation}
\label{eq:8a}
W_{T_1 \cdot T_2} \silo W_{T_1} \verk W_{T_2} \; .
\end{equation}
Let us explain how this follows from Theorem \ref{t3}. For $S \peq T_1$ set
\[
M_S = W_{S \cdot T_2} \; .
\]
For $n \in S$ we have $(S \cdot T_2) / n = (S/n) \cdot T_2$. Hence the usual Witt vector Frobenius and Verschiebung morphisms give morphisms
\[
F_n : M_S \longrightarrow M_{S/n} \quad \text{and} \quad V_n : M_{S/n} \longrightarrow M_S \; .
\]
Using them we obtain an object $\ubM = (M_S)_{S \peq T_1}$ of $\Rh\Fh\Vh_{T_1/\Z}$ with $M := M_{ \{ 1 \} } = W_{T_2}$. The coordinate rings $B_S$ of all $M_S$ are polynomial algebras over $\Z$ and so are $B_S \otimes_{\Z} B_S$. The ring $W_{T_2} (A)$ is a subring of $A^{T_2}$ via the ghost map if $A$ has no $T_2$-torsion. Therefore if $A$ has no $\Z$-torsion it follows that $\ubM$ has no Hopf $T_1$-torsion. Hence we can apply Theorem \ref{t3} b) and obtain a uniquely determined isomorphism
\[
\alpha : \ubM = (W_{S \cdot T_2})_{S \peq T_1} \longrightarrow \ubW^M = (W_S \verk W_{T_2})_{S \peq T_1}
\]
in $\Rh \Fh \Vh_{T/\Z}$ with $\alpha_S = \id$ for $S = \{ 1 \}$. In particular we get a natural isomorphism
\begin{equation}
\label{eq:8b}
\alpha_{T_1} : W_{T_1 \cdot T_2} \silo W_{T_1} \verk W_{T_2} \; .
\end{equation}
The Witt vector functor $W_S$ transforms short exact sequences
\[
0 \longrightarrow I \longrightarrow A \longrightarrow B \longrightarrow 0 
\]
of rings into short exact sequences of rings
\[
0 \longrightarrow W_S (I) \longrightarrow W_S (A) \longrightarrow W_S (B) \longrightarrow 0 \; .
\]
Hence the isomorphism \eqref{eq:8b} is also valid on non-unital rings (they are kernels of surjections of unital rings). Applying \eqref{eq:8b} to $A^{(q)}$ for $\Z [q]$-algebras $A$ we obtain an isomorphism of ring-schemes over $\spec \Z [q]$
\[
W^{(q)}_{T_1 \cdot T_2} \silo W_{T_1} \verk W^{(q)}_{T_2} \; .
\]
Note that for $T_2 = \{ 1 \}$ this holds by definition. 

\section{Deformations of Witt vector schemes} \label{sec:5}
In this section we prove universality over reduced bases for the $q$-deformation $W^{(q)}$ introduced in section \ref{sec:2neu} and show that $W^{(q)}$ is isomorphic to $\oW^{1-q}$. 

For an $R$-algebra $C$ and an element $r \in R$ we denote by $C^{(r)}$ the ring with underlying additive group $(C , +)$ and twisted multiplication $x * y = r xy$. The ring $C^{(r)}$ is unital if and only if $r \in R^{\times}$. In this case $1_{C^{(r)}} = r^{-1}$ is the unity. Let $\G^{(r)}_{R}$ be the ring-scheme over $R$ defined by $\G^{(r)}_R (C) = C^{(r)}$ for all $R$-algebras $C$. It is represented by the polynomial algebra $R [t]$ together with 
\[
\Delta_+ (t) = t \otimes 1 + 1 \otimes t \quad \text{and} \quad \Delta_{\hullet} (t) = rt \otimes t \;  . 
\]
Moreover $\varepsilon_0 (t) = 0$ and $\G^{(r)}_R$ is unital if and only if $r \in R^{\times}$. In this case the co-unit $\varepsilon_1$ is given by $\varepsilon_1 (t) = r^{-1}$. 

In particular we have a non-unital ring-scheme $\G^{(q)}$ over $\Z [q]$ and a unital ring-scheme $\G^{(q)}$ over $\Z [q , q^{-1}]$. An element $r \in R$ corresponds to a ring homomorphism $\Z [q] \to R$ and we have
\begin{equation}
\label{eq:6}
\G^{(r)}_{R} = \G^{(q)} \otimes_{\Z [q]} R
\end{equation}
as non-unital ring-schemes over $R$. Similarly a unit $r \in R^{\times}$ determines a ring homomorphism $\Z [q , q^{-1}] \to R$ and
\[
\G^{(r)}_{R} = \G^{(q)} \otimes_{\Z [q, q^{-1}]} R
\]
as unital ring-schemes. For any two units $r , r' \in R^{\times}$ the ring-schemes $\G^{(r)}_R$ and $\G^{(r')}_R$ are isomorphic in the category of unital ring-schemes. Now let $r , r' \in R$ be arbitrary. In the category of non-unital ring-schemes $\G^{(r)}_R$ and $\G^{(r')}_R$ are isomorphic if and only if $r' = ru$ for some $u \in R^{\times}$. There is a natural bijection
\begin{equation}
\label{eq:7}
\{ u \in R^{\times} \mid r = r'u \} \silo \Iso (\G^{(r)}_R , \G^{(r')}_R)
\end{equation}
and in particular an isomorphism of groups
\[
\{ u \in R^{\times} \mid r = ru \} \silo \Aut (\G^{(r)}_R) \; .
\]
Here $\Iso$ and $\Aut$ are taken in the category of non-unital ring-schemes over $R$. If $r \in R$ is not a zero-divisor then $\Aut (\G^{(r)}_R) = \{ \id \}$. 

We denote by $(\G^{(r)}_R , t)$ the pair consisting of $\G^{(r)}_R$ and the chosen local coordinate $t$ above. Consider two ring-schemes $M_1$ and $M_2$ over $R$ whose underlying schemes are isomorphic to $\A^1$ and which are enhanced by coordinates $z_1$ and $z_2$ in zero. We say that $(M_1 , z_1)$ and $(M_2, z_2)$ are isomorphic if there is an isomorphism (necessarily unique) of ring-schemes $\varphi : M_1 \to M_2$ with $z_1 = z_2 \verk \varphi$. Clearly $(\G^{(r)}_R , t) \cong (\G^{(r')}_R , t)$ if and only if $r = r'$. 

\begin{prop}
\label{t5}
Let $R$ be a reduced ring and let $M$ be a ring-scheme over $R$ with underlying scheme isomorphic to $\A^1$ and coordinate $z$ in zero. Then there exists a unique element $r \in R$ such that $(M, z) \cong (\G^{(r)}_R , t)$. If $M$ is unital and we work in the category of unital ring-schemes, then $r \in R^{\times}$. 
\end{prop}

\begin{proof} We have $M = \spec R [z]$. The ring-scheme structure of $M$ corresponds to co-addition $\Delta_+$, co-multiplication $\Delta_{\hullet}$, co-zero $\varepsilon_0$ and in the unital case a co-unit $\varepsilon_1$. By the choice of $z$ we have $\varepsilon_0 (z) = 0$. Set $F (X,Y) = \Delta_+ (z)$ and $G (X,Y) = \Delta_{\hullet} (z)$ where we set $X = z \otimes 1$ and $Y = 1 \otimes z$. Then $F,G$ determine a one-dimensional polynomial ring law over $R$. Now the assertion follows from the next lemma.
\end{proof}

\begin{lemma}
\label{t6}
Let $R$ be a reduced ring and $F,G \in R [X,Y]$ a one-dimensional polynomial ring law over $R$. Then there is a unique element $r \in R$ such that
\[
F (X,Y) = X + Y \quad \text{and} \quad G (X,Y) = rXY \; .
\]
\end{lemma}

\begin{proof}
Let $n \ge 1$ be the highest power of $x$ in $F (x,y)$. An inspection of the formula
\[
F (x , F (y,z)) = F (F (x,y) , z)
\]
shows that $n = n^2$ since $R$ has no nilpotent elements except $0$. Hence $n = 1$ and similarly for $y$. Thus
\[
F (x,y) = x+y + cxy \quad \text{for some} \; c \in R \; .
\]
By assumption there is a polynomial $I (x) = -x + \deg \ge 2$ such that $F (x, I (x)) = 0$. Since $R$ is reduced, it follows that $c = 0$. Hence $F (x,y) = x+y$. Associativity of $G$ implies similarly as for $F$ that $G (x,y) = rxy$ for some $r \in R$. \quad
\end{proof}

\begin{rem}
Over non-reduced bases there are more one-dimensional polynomial ring laws than the ones in the lemma. Over $R = \Z [\varepsilon] / (\varepsilon^2)$ for example the polynomials $F (X,Y) = X + Y + \varepsilon XY$ and $G (X,Y) = r \varepsilon XY$ for $r \in R$ define a different class.
\end{rem}

Since we are interested in deformations of Witt vector schemes it is natural to introduce the full subcategory $\Ah \Fh \Vh_{T / R}$ of $\Rh \Fh \Vh_{T/R}$ whose objects $\ubM$ satisfy $M_{ \{ 1 \} } \cong \A^1$ as schemes. We rigidify $\Ah \Fh \Vh_{T/R}$ by considering the category $\Ah \Fh \Vh^+_{T/R}$ of pairs $(\ubM , z)$ where $z$ is a coordinate in zero of $M_{ \{ 1 \} }$. A morphism $\varphi : (\ubM , z) \to (\ubM' , z')$ is a morphism $\varphi : \ubM \to \ubM'$ in $\Ah \Fh \Vh^+_{T/R}$ such that $z = z' \verk \varphi_{\{ 1 \} }$. In particular $\varphi_{ \{ 1 \} } : M_{ \{ 1 \} } \to M'_{ \{ 1 \} }$ is an isomorphism.

For $S \peq T$ recall the non-unital ring scheme $W^{(q)}_S$ over $\Z [q]$ from section \ref{sec:2neu}. We have $W^{(q)}_S = W^{\G^{(q)}}_S$ and $W^{(q)}_{\{ 1 \} } = \A^1$ is enhanced by the fixed coordinate $t$. Thus we may view $\ubW^{(q)} := \ubW^{\G^{(q)}}$ as objects of $\Ah \Fh \Vh_{T/ \Z [q]}$ and $\Ah \Fh \Vh^+_{T / \Z [q]}$. Similar definitions over the ring $\Z [q , q^{-1}]$ hold in the unital case.

\begin{cor}
\label{t7neu}
Let $R$ be a reduced ring without $T$-torsion and $\ubM$ an object of $\Ah \Fh \Vh^+_{T/R}$. In the non-unital case there is a unique homomorphism $\Z [q] \to R$ such that there exists a morphism in $\Ah \Fh \Vh^+_{T/R}$
\[
\alpha : \ubM \longrightarrow \ubW^{(q)} \otimes_{\Z [q]} R \; .
\]
The morphism $\alpha$ is uniquely determined and it is an isomorphism. In the unital case the corresponding assertions hold with $\Z [q]$ replaced by $\Z [q , q^{-1}]$. 
\end{cor}

\begin{proof}
Since the sequence \eqref{eq:7a} is scheme-theoretically split and since $M = M_{ \{ 1 \} }$ is isomorphic to $\A^1$, it follows inductively that $M_S \cong \A^S$ as schemes if $S$ is finite. Continuity of $\ubM$ implies that $M_S \cong \A^S$ for all $S \peq T$. Writing $M_S = \spec B_S$, the algebra $B_S$ is therefore a polynomial algebra over $R$. Since $R$ has no $T$-torsion by assumption and since $M \cong \A^1$ we see that $\ubM$ has no Hopf $T$-torsion. Now the corollary follows from Theorem \ref{t3} and Proposition \ref{t5}.
\end{proof}

\begin{rem}
For fixed $T$ it follows that $ (\ubW^{(q)} , t)$ over $\spec \Z [q]$ (resp. $\spec \Z [q , q^{-1}]$) is the universal deformation over reduced bases of $(\ubW , t)$ in $\Ah \Fh \Vh^+_T$. 
\end{rem}

We say that a ring homomorphism $\varphi : \Z [q] \to R$ is determined up to a unit in $R$ if $\varphi (q)$ is determined up to a unit in $R$. Forgetting the coordinate of $M$ we get

\begin{cor}
\label{t7}
Let $R$ be a reduced ring without $T$-torsion and $\ubM$ an object of the category $\Ah\Fh\Vh_{T/R}$. \\
a) In the non-unital case there is a homomorphism $\varphi : \Z [q] \to R$ which is determined up to a unit in $R$ such that there is an isomorphism in $\Rh\Fh\Vh_{T/R}$
\[
\alpha : \ubM \to \ubW^{(q)} \otimes_{\Z [q] , \varphi} R \; .
\]
If $\varphi (q)$ is not a zero-divisor in $R$ then $\alpha$ is uniquely determined. In the general case the set of $\alpha$'s is parametrized by the stabilizer of $\varphi (q)$ under the action of $R^{\times}$ on $R$.\\
b) In the unital case there is a unique isomorphism in $\Ah\Fh\Vh_{T/R}$
\[
\alpha : \ubM \longrightarrow \ubW \otimes_{\Z} R \; .
\]
Thus $\ubW$ has no non-trivial deformations over reduced bases in $\Ah\Fh \Vh_T$. 
\end{cor}

\begin{exmp} \label{t-ex1} \rm
Recall the second family $\oW^{g (q)}_{\!\!S}$ of $q$-deformed Witt vector schemes over $\spec \Z [q]$ from section \ref{sec:2neu}. Together with their Frobenius and Verschiebung structures and the first coordinate they make
\[
\overline{\ubW}^{g(q)} := (\oW^{g(q)}_{\!\!S})_{S \peq T}
\]
an object of $\Ah\Fh\Vh^+_{T / \Z [q]}$. We have $\oW^{g (q)}_{\{ 1 \} } = \G^{(1 - g (q))}$ by construction and this identification respects the preferred coordinates at zero of both sides. 
\end{exmp} 

It follows from Corollary \ref{t7neu} that there is a unique isomorphism 
\[
\overline{\ubW}^{g(q)} \silo \ubW^{(1-g (q))} \quad \text{in} \quad \Ah \Fh \Vh^+_{T / \Z [q]} \; .
\]
Moreover this is the only isomorphism of $\overline{\ubW}^{g (q)}$ with $\ubW^{(r)}$ for some $r \in \Z [q]$ in this category. Forgetting the coordinate at zero and working in the category $\Ah \Fh \Vh_{T / \Z [q]}$, Corollary \ref{t7} implies that there are exactly two values of $r \in \Z [q]$ for which an isomorphism (automatically unique) $\overline{\ubW}^{g (q)} \silo \ubW^{(r)}$ exists, namely $r = 1 - g (q)$ and $r = g (q) -1$. Note here that $\Z [q]$ is an integral domain with $\Z [q]^{\times} = \{ \pm 1 \}$. In particular it follows that $\oW^{g (q)}_{\!\!S} \cong \oW^{2 - g (q)}_{\!\!S}$ as observed in \cite{O2} Proposition 3.9.

\begin{exmp} \label{t-ex2} \rm
In \cite{L} Lenart introduced modified Witt rings $W^q (A)$ for every integer $q \in \Z$. In \cite{O1} truncated versions $W^q_S (A)$ of these rings were studied. They are obtained as follows. Define $W^q_S (A) = A^S$ as sets and consider the ghost map
\[
\Phi^q_S : W^q_S (A) \longrightarrow A^S \; ,
\]
given by the formula
\[
\Phi^q_S ((a_d)_{d \in S}) = \Big( \sum_{d \mid n} d q^{\frac{n}{d}-1} a^{\frac{n}{d}}_d \Big)_{n \in S} \; .
\]
It is the same ghost map as for $W^{(q)}_S (A)$ in formula \eqref{eq:1n} but on the ghost side $A^S$ is taken componentwise with the usual ring-structure (instead of $(A^{(q)})^S$ as in \eqref{eq:1n}). Using Fermat's little theorem it is shown in \cite{L}, \cite{O1} that for any $q \in \Z$ there is a unique, possibly non-unital ring-structure on $W^q_S (A)$ which is functorial in $A$ such that $\Phi^q_S$ becomes a ring homomorphism. Moreover the usual Frobenius and Verschiebung operators for $n \in S$ on the ghost side
\[
F_n ((a_{\nu})_{\nu \in S}) = (a_{n \nu})_{\nu \in S/n} \quad \text{and} \quad V_n ((a_{\nu})_{\nu \in S/n}) = (n \delta_{n \mid \nu} a_{\nu / n})_{\nu \in S}\; ,
\]
induce corresponding morphisms on the Lenart--Witt side. It is immediate that all conditions for $\ubW^q := (W^q_S)_{S \peq T}$ to be an object of $\Ah \Fh \Vh^+_{T / \Z}$ are satisfied except possibly for the property that $F_p$ should reduce $\mod p$ to the $p$-th power map for all primes $p \in S \peq T$. It can be shown with some effort that this condition is satisfied if and only if no prime divisor of $q$ is contained in $T$. Since $W^q_{\{ 1 \} } (A) = A$ as rings for any $q \in \Z$, it follows from Corollary \ref{t7neu} that $\ubW^q$ is uniquely isomorphic to $\ubW$ in $\Ah \Fh \Vh^+_{T / \Z}$ if no prime divisor of $q$ is in $T$. Thus in this case there are natural isomorphisms of rings
\[
W^q_S (A) \silo W_S (A) \; .
\]
This fact is a special case of \cite{O1} Theorem 14. If $T$ contains a prime divisor of $q$ then our theory does not apply to $\ubW^q$. Let us illustrate the preceding discussion with the case $T = \{ 1 ,p \}$ where the calculations are easy. 

The ghost maps are $\Phi^q_{ \{ 1 \} } = \id$ and
\[
\Phi^q_{ \{ 1, p \} } (a_1, a_p) = (a_1 , p a_p + q^{p-1} a^p_1) \; .
\]
Hence $W^q_{ \{ 1 \} } (A) = A$ as rings and on $W^q_{ \{ 1 , p \} } (A)$ addition and multiplication are given as follows
\[
(a_1 , a_p) + (b_1 , b_p) = \Big( a_1 + b_1 , a_p + b_p - q^{p-1} \sum^{p-1}_{\nu = 1} \frac{1}{p} {p \choose \nu} a^{\nu}_1 b^{p- \nu}_1 \Big)
\]
and
\[(a_1 , a_p)\cdot (b_1 , b_p) = (a_1 b_1 , pa_p b_p + q^{p-1} (a_p b^p_1 + a^p_1 b_p) + \frac{1}{p} q^{p-1} (q^{p-1} - 1) a^p_1 b^p_1) \; .
\]
The Frobenius morphism
\[
F_p : W^q_{ \{ 1,p \} } (A) \longrightarrow W^q_{ \{ 1 \} } (A) = A
\]
is given by the formula
\[
F_p (a_1 , a_p) = pa_p + q^{p-1} a^p_1 \; .
\]
With the projection
\[
\pi : W^q_{ \{ 1 ,p \} } (A) \longrightarrow W^q_{ \{ 1 \} } (A) = A \; , \; a = (a_1 , a_p) \longmapsto a_1
\]
we therefore have
\[
F_p (a) \equiv q^{p-1} \pi (a)^p \mod p A \; .
\]
Thus the relation
\[
F_p (a) \equiv \pi (a)^p \mod p A
\]
for all rings $A$ is equivalent to $q^{p-1} \equiv 1 \mod p$ i.e. to the assertion that $p \in T = \{ 1 , p \}$ is not a prime divisor of $q$. In this case the map
\[
\alpha : W^q_{ \{ 1 ,p \} } (A) \silo W_{ \{1,p \} } (A)
\]
with
\[
\alpha (a_1 , a_p) = \Big( a_1 , a_p + \frac{q^{p-1}-1}{p} a^p_1 \Big)
\]
is an isomorphism of rings. 
\end{exmp}

\begin{small}
\section*{Appendix: Witt vectors of inductive systems of rings}

We sketch a natural generalization of the theory of Witt vectors to $\ind$-rings.

Let $S \subset \N$ be a divisor stable subset and $\ubA = (A_n)_{n \in S}$ an inductive system of unital or non-unital commutative rings. This means that for $n \in S$ and $d \mid n$ there are ring homomorphisms
\[
\pi_{d,n} : A_d \to A_n
\]
with $\pi_{n,n} = \id$ and $\pi_{d_1,n} = \pi_{d,n} \verk \pi_{d_1,d}$ if $d_1 \mid d$.

Consider the set
\[
W_S (\ubA) = \prod_{n \in S} A_n
\]
and the ghost map
\[
\Gh_S : W_S (\ubA) \longrightarrow \prod_{n \in S} A_n
\]
defined by
\[
\Gh_{S,n} ((a_{\nu})_{\nu \in S}) = \sum_{d\mid n} d \pi_{d,n} (a_d)^{n/d} \quad \text{in} \; A_n \; .
\]

\begin{prop} \label{t9}
Assume that $A_n$ has no $n$-torsion for every $n \in S$. Then \\
$\Gh_S : W_S (\ubA) \to \prod_{n \in S} A_n$ is injective.
\end{prop}

\begin{proof}
Assume that
\[
\Gh_S ((a_{\nu})) = \Gh_S ((b_{\nu})) \; .
\]
By definition we get $a_1 = b_1$. For $n \in S , n \neq 1$ assume that $a_d = b_d$ in $A_d$ has been shown for all $d \propdiv n$. The equation $\Gh_{S,n} ((a_{\nu})) = \Gh_{S,n} ((b_{\nu}))$ gives
\[
n (a_n - b_n) = \sum_{d \propdiv n} d (\pi_{d,n} (b_d)^{n/d} - \pi_{d,n} (a_d)^{n/d}) = 0
\]
and hence $a_n = b_n$ since $A_n$ has no $n$-torsion.
\end{proof}

\begin{example}
If $(x_n) = \Gh_S ((a_{\nu}))$ and $p \in S$, then
\[
x_p = \pi_{1,p} (a_1)^p + p a_p \quad \text{and hence} \; x_p \equiv \pi_{1,p} (a_1)^p \mod p A_p \; .
\]
The Witt polynomials $\Sigma_n$ and $\Pi_n$ for addition and multiplication depend only on the variables $x_d$ with $d \mid n$. Hence we can define addition and multiplication on $W_S (\ubA)$ by setting
\[
\bfa \oplus \bfb = \bfc \quad \text{for} \; \bfa , \bfb \in W_S (\ubA)
\]
where 
\[
c_n = \Sigma_n (\pi_{d,n} (a_d) , \pi_{d,n} (b_d) ; d \mid n)
\]
and similarly for multiplication. As in the usual case this is the only ring structure on the set $W_S (\ubA)$ which is functorial in $\ubA$ and for which $\Gh_S$ is a ring homomorphism if $\prod_{n \in S} A_n$ is equipped with componentwise addition and multiplication. Similarly the polynomials defining the usual Frobenius morphisms $F_n : W_S \to W_{S/n}$ define commuting ring homomorphisms\\
$F_n : W_S (\ubA) \to W_{S/n} (F_n (\ubA))$ where $F_n (\ubA) = (A_{n\nu})_{\nu \in S/n}$. There are also additive Verschiebung maps 
\[
V_n : W_{S/n} (V_n (\ubA)) \to W_S (\ubA)
\]
where $V_n (\ubA) = (A_{\nu})_{\nu \in S/n}$ and $V_n ((a_{\nu})_{\nu \in S/n}) = (\delta_{n \mid \mu} \pi_{\mu / n , \mu} (a_{\mu / n}))_{\mu \in S}$ with $\delta_{n\mid \mu} = 1$ if $n \mid \mu$ and $\delta_{n \mid \mu} = 0$ if $n \nmid \mu$. The projection 
\[
\res : W_S (\ubA) = \prod_{n \in S} A_n \overset{\pr_1}{\twoheadrightarrow} A_1
\]
is a surjective homomorphism of rings. 
\end{example}

\begin{example}
A ring $A$ can be viewed as the inductive system $A_{\const}$ where $A_n = A$ and $\pi_{d,n} = \id$ for $d \mid n , n \in S$. Then $W_S (A_{\const}) = W_S (A)$ together with Frobenius and Verschiebung maps. We can also form the trivial inductive system $A_{\triv}$ of non-unital rings with $A_n = A$ and $\pi_{n, n} = \id , \pi_{d,n} = 0$ for $d \propdiv n$. Let $A^{(n)} = A$ as an additive group but with ring structure $a \ast b = n ab$. Then $W_S (A_{\const}) = \prod_{\nu \in S} A^{(\nu)}$, the product ring. The Frobenius map $F_n$ corresponds to the map
\[
F_n : \prod_{\nu \in S} A^{(\nu)} = W_S (A_{\const}) \to W_S (F_n A_{\const}) = \prod_{\nu \in S / n} A^{(\nu)}
\]
with
\[
F_n ((a_{\nu})_{\nu \in S}) = (n a_{\nu n})_{\nu \in S / n}\; .
\]
The ring $W_{\N} (A_{\const})$ appeared previously as $W^0 (A)$ in \cite{L} Corollary 5.9 (1). 

For a divisor stable subset $T \subset S$ and an inductive system $\ubA = (A_n)_{n \in S}$ there is a natural surjective ring homomorphism
\[
W_S (\ubA) \xrightarrow{\proj} W_T (\res^T_S (\ubA)) \; , \; (a_n)_{n \in S} \mapsto (a_n)_{n \in T}
\]
where $\res^T_S (\ubA) = (A_{\nu})_{\nu \in T}$.
\end{example}

\begin{prop} \label{t12}
Consider an inductive system $\ubA = (A_n)_{n \in S}$. Then the following diagram is commutative:
\[
\xymatrix{
W_S (\ubA) \ar[r]^-{F_p} \ar[d] & W_{S/p} (F_p (\ubA)) \ar[r] & W_{S / p} (F_p (\ubA)) / p \\
W_S (\ubA)/p \ar[r]^{(\,)^p} & W_S (\ubA) / p \ar[r] & W_{S / p} (\res^{S/p}_S (\ubA)) / p \ar[u]_{W_{S/p} (\pi)}
}
\]
Here $\pi : \res^{S/p}_S (\ubA) \to F_p (\ubA)$ is the map with $\nu$-th component $\pi_{\nu , p\nu} : A_{\nu} \to A_{p \nu}$ for $\nu \in S / p$.
\end{prop}

\begin{proof}
Set $\Lambda = \Z [x_d \mid d \in S]$ and let $(f_{\nu})_{\nu \in S/p}$ with $f_{\nu} \in \Z [x_d \mid d | p\nu]\subset \Lambda$ be the family of polynomials defining the morphism $F_p : W_S \to W_{S/p}$. Also let $(G_{\mu})_{\mu \in S}$ with $G_{\mu} \in \Z [x_d \mid d|\mu] \subset \Lambda$ be the family of polynomials defining the $p$-th power morphism $(\,)^p : W_S \to W_S$. It is known that for $\nu \in S / p$ the difference $f_{\nu} - G_{\nu}$ is divisible by $p$ in $\Lambda$. This is equivalent to the fact that diagram \eqref{eq:1a} commutes. Now for $(a_{\mu})_{\mu \in S} \in W_S (\ubA)$ the $\nu$-th component for $\nu \in S/p$ of 
\[
F_p ((a_{\mu})_{\mu \in S}) - W_{S/p} (\pi) (\pi_{S , S/p} ((a_{\mu})^p_{\mu \in S}))
\]
is given by
\[
f_{\nu} (\pi_{d, p\nu} (a_d) ; d \mid p\nu) - G_{\nu} (\pi_{d , p\nu} (a_d)) ; d \mid p \nu) \; .
\]
This holds because the transition maps are ring homomorphisms and $\pi_{\nu , p\nu} \verk \pi_{d, \nu} = \pi_{d , p\nu}$. The assertion follows.
\end{proof}

There is a Dwork lemma for our Witt vector rings of inductive systems.

{\bf Dwork lemma} Let $\ubA = (A_n)_{n \in S}$ be an inductive system of rings on a divisor stable subset $S$ of $\N$. Assume that for all $n \in S$ and primes $p \mid n$ we are given ring homomorphisms (compatible with the transition maps)
\[
\phi_p : A_{n/p} \to A_n
\]
such that the following diagram commutes:
\[
\xymatrix{
A_{n/p} \ar[r]^{\phi_p} \ar[d] & A_n \ar[r] & A_n / p \ar@{=}[d] \\
A_{n/p}/ p \ar[r]^{( \, )^p} & A_{n/p} / p \ar[r]^{\pi_{n/p,n}} & A_n / p\; .
}
\]
Thus, for $a \in A_{n/p}$ we have in $A_n$
\[
\phi_p (a) \equiv \pi_{n/p,n} (a^p) \mod p A_n \; .
\]
Then the image of the ghost map
\[
\Gh_S : W_S (\ubA) \to \prod_{n \in S} A_n
\]
is the following subring:
\[
\Gh_S (W_S (\ubA)) = \{ (x_n)_{n \in S} \mid \phi_p (x_{n/p}) \equiv x_n \mod p^{v_p (n)} A_n \; \text{for} \; p \mid n , n \in S \} \; .
\]

\begin{proof}
The argument in the proof of Lemma 1.1 in \cite{H} can be easily adapted to our setting. Let us look at the case $S = \{ 1, p \}$ for example. If $(x_n) = \Gh_S ((a_{\nu}))$, then $x_1 = a_1$ and $x_p = \pi_{1,p} (x^p_1) + pa_p$. Hence $(x_1, x_p) \in \Gh_S (W_S (\ubA))$ if and only if $x_p \equiv \pi_{1,p} (x^p_1) \mod p A_p$ or equivalently, if and only if $x_p \equiv \phi_p (x_1) \mod p A_p$. 
\end{proof}

We need the following construction. Given an inductive system of rings $\ubA = (A_n)_{n \in S}$, we obtain another such system $\ubW (\ubA)$ by setting
\[
\ubW (\ubA)_n = W_{S/n} (F_n (\ubA)) \quad \text{for} \; n \in S
\]
and defining $\pi_{d,n}$ to be the composition
\[
\pi_{d,n} : W_{S/d} (F_d (\ubA)) \xrightarrow{\proj} W_{S/n} (\res^{S/n}_{S/d} (F_d (\ubA))) \xrightarrow{W_{S/n} (\pi)} W_{S/n} (F_n (\ubA)) \; .
\]
Here
\[
\pi : \res^{S/n}_{S /d} (F_d (\ubA)) = (A_{\nu d})_{\nu \in S/n} \to (A_{\nu n})_{\nu \in S/n} = F_n (\ubA)
\]
is the map with $\nu$-th component $\pi_{\nu d , \nu n} : A_{\nu d} \to A_{\nu n}$. It follows from Proposition \ref{t12} that the maps $F_p$ equip $\ubW (\ubA)$ with a commuting family of Frobenius lifts as in the Dwork lemma. The commutation property of the $F_p$ follows from the known commutation properties of the universal polynomials defining the Witt vector Frobenius morphisms. The surjective ring homomorphisms
\[
\ubW (\ubA)_n = W_{S/n} (F_n (\ubA)) \xrightarrow{\res} F_n (\ubA)_1 = A_n
\]
are compatible with the transition maps $\pi_{d,n}$ for $d \mid n , n \in S$. Hence we obtain a map of inductive systems of rings
\[
\res : \ubW (\ubA) \to \ubA \; .
\]
The ghost maps
\[
\Gh_{S/n} : W_{S/n} (F_n (\ubA)) \to \prod_{k \in S/n} A_{kn}
\]
combine into a morphism of inductive systems of rings
\[
\ubGh : \ubW (\ubA) \to (\prod_{k \in S/n} A_{kn})_{n \in S}
\]
such that $\res \verk \ubGh = \res$.

{\bf Universal property of $\ubW$} In the situation of the Dwork lemma, assume in addition that $A_n$ has no $n$-torsion for each $n \in S$. Moreover we suppose that $\phi_p$ commutes with $\phi_l$ for all primes $p , l \in S$. This means that for every $n \in S$ with $p \mid n$ and $l \mid n$ the following diagram commutes:
\[
\xymatrix{
A_{n/p l} \ar[r]^{\phi_p} \ar[d]_{\phi_l} & A_{n/l} \ar[d]^{\phi_l} \\
A_{n/p} \ar[r]^{\phi_p} & A_n \; .
}
\]
Then there is a unique morphism $\lambda : \ubA \to \ubW (\ubA)$ of inductive systems of rings with the following properties:\\
a) $\res \verk \lambda = \id_{\tA}$\\
b) $\lambda$ commutes with the Frobenius lifts on $\ubA$ and $\ubW (\ubA)$, i.e. the diagrams
\[
\xymatrix{
A_{n/p} \ar[r]^-{\lambda_{n/p}} \ar[d]_{\phi_p} & \ubW (\ubA)_{n/p} \ar[d]^{F_p} \\
A_n \ar[r]^-{\lambda_n} & \ubW (\ubA)_n
}
\]
commute for all $p \mid n , n \in S$.

\begin{proof}
We first assume that $\lambda$ satisfying the desired properties exists and determine its form. According to the Dwork lemma and Proposition \ref{t9}, via the injective ghost map we have isomorphisms for all $n \in S$ 
\[
\begin{array}{l}
\ubW (\ubA)_n = W_{S/n} (F_n(\ubA)) \cong \\[0.2cm]
\{ (x_k) \in \prod_{k \in S/n} A_{nk} \mid \phi_l (x_{k/l}) \equiv x_k \mod l^{v_l (k)} A_{nk} \; \text{for}\; l \mid k , k \in S/n \}\; ,
\end{array} 
\]
where $l$ denotes prime numbers. In the following proof, we will view this isomorphism as an identification.

For $m = p^{\nu_1}_1 \cdots p^{\nu_r}_r$ dividing $n \in S$ we set
\[
\phi_m := \phi^{\nu_1}_{p_1} \verk \ldots \verk \phi^{\nu_r}_{p_r} : A_{n/m} \to A_n \; .
\]
By assumption, this is independent of the ordering of the primes $p_1 , \ldots , p_r$ dividing $m$. For the $\ubW$-Frobenius morphisms the corresponding formula
\[
F_m = F^{\nu_1}_{p_1} \verk \ldots \verk F^{\nu_r}_{p_r} : \ubW (\ubA)_{n/m} \to \ubW (\ubA)_n
\]
holds and we therefore have a commutative diagram for $m \mid n , n \in S$
\[
\xymatrix{
A_{n/m} \ar[r]^-{\lambda_{n/m}} \ar[d]_{\phi_m} & \ubW (\ubA)_{n/m} \ar[d]^{F_m} \\
A_n \ar[r]^-{\lambda_n} & \ubW (\ubA)_n \; .
}
\]
In the representation of $\ubW (\ubA)_{n/m}$ and $\ubW (\ubA)_n$ on the ghost side via the Dwork lemma, the map $F_m$ is given by $F_m ((x_k)_{k \in S / (n/m)}) = (x_{km})_{k \in S/n}$. For $a \in A_{n/m}$ consider the relation
\[
\lambda_n (\phi_m (a)) = F_m (\lambda_{n/m} (a)) \; .
\]
The property $\res \verk \lambda = \id$ implies that $\lambda_n (b)_1 = b$ for all $b \in A_n$. Hence
\[
\phi_m (a) = \lambda_n (\phi_m (a))_1 = \lambda_{n/m} (a)_m \; .
\]
Setting $\nu = n/m$, it follows that for all $\nu \in S , m \in S/ \nu$ and $a \in A_{\nu}$, we have $\lambda_{\nu} (a)_m = \phi_m (a)$, i.e. $\lambda_{\nu} (a) = (\phi_m (a))_{m \in S/\nu}$. Thus we have seen that a map $\lambda : \ubA \to \ubW (\ubA)$ with properties a) and b) must have the form $\lambda = (\lambda_n)_{n \in S}$ where $\lambda_n : A_n \to \ubW (\ubA)_n = W_{S /n } (F_n \ubA)$ is given by $\lambda_n (a) = (\phi_k (a))_{k \in S/n}$. On the other hand, setting $x_k = \phi_k (a)$ for $a \in A_n$ we have
\[
\phi_l (x_{k/l}) = \phi_l \phi_{k/l} (a) = \phi_k (a) = x_k \; .
\]
Hence $\lambda_n (a) = (\phi_k (a))_{k \in S/n}$ is indeed an element of $\ubW (\ubA)_n$ in the Dwork lemma description above. Clearly $\lambda_n$ so defined is a ring homomorphism with 
\[
\lambda_n (a)_1 = \phi_1 (a) = a \; \text{i.e.} \; \res \verk \lambda = \id\; .
\]
The maps $\lambda_n$ are compatible with the transition maps of $\ubA$ and $\ubW (\ubA)$ because the $\phi_p$ are compatible with transition maps by definition. Finally, we have
\[
\lambda_n (\phi_p (a)) = (\phi_k (\phi_p (a)))_{k \in S/n} = (\phi_{kp} (a))_{k \in S/n} = F_p ((\phi_k (a)))_{k \in S/(n/p)}
\]
i.e. $\lambda_n \verk \phi_p = F_p \verk \lambda_{n/p}$. 
\end{proof}

\begin{cor}[Universal property of $\ubW$ over $S$]
Fix a divisor stable set $S$ and consider inductive systems of rings $\ubA = (A_n)_{n \in S}$ and $\ubB = (B_n)_{n \in S}$ with the following properties:\\
i) $A_n$ and $B_n$ have no $n$-torsion for all $n \in S$.\\
ii) $\ubA$ is equipped with commuting Frobenius lifts.\\
Then for any morphism $\alpha : \ubA \to \ubB$ of $\ind$-rings there is a unique morphism $\beta : \ubA \to \ubW (\ubB)$ of $\ind$-sets commuting with the Frobenius maps such that the diagram
\[
\xymatrix{
 & \ubW (\ubB) \ar[dr]^{\res} & \\
\ubA \ar[ur]^{\beta} \ar[rr]^{\alpha} & & \ubB
}
\]
commutes. The morphism $\beta$ is a morphism of $\ind$-rings. 
\end{cor}

\begin{proof}
The existence follows from the Catier--Dieudonn\'e lemma by setting $\beta = \ubW (\alpha) \verk \lambda$ and looking at the commutative diagram
\[
\xymatrix{
A \ar[r]^-{\lambda} \ar@{=}[dr] & \ubW (\ubA) \ar[r]^{\ubW (\alpha)} \ar[d]^{\res} & \ubW (\ubB) \ar[d]^{\res} \\
 & \ubA \ar[r]^{\alpha} & \ubB \; .
}
\]
For the uniqueness, assume that we are given another lift (of $\ind$-sets) $\beta' : \ubA \to \ubW (\ubB)$ of $\alpha$ commuting  with the Frobenius maps. Since the $B_n$ have no $n$-torsion the ghost map $\ubW (\ubB) \to (\prod_{k \in S/n} B_{kn})_{n \in S}$ has injective components.

Hence it suffices to show that a map $\gamma$ making the following diagram commutative is uniquely determined if it commutes with the Frobenius maps
\[
\xymatrix{
 & \Big( \prod_{k \in S/n} B_{kn} \Big)_{n \in S} \ar[d]^{\pi} \\
\ubA \ar[ur]^{\gamma} \ar[r]^{\alpha} & \ubB \; .
}
\]
Here $\pi = (\pi_n)_{n \in S}$ where $\pi_n : \prod_{k \in S/n} B_{kn} \to B_n$ maps $(x_k)_{k \in S/n}$ to $x_1$. Moreover, commutation with Frobenius maps means that all diagrams
\[
\xymatrix{
A_{n/p} \ar[r]^-{\gamma_{n/p}} \ar[d]_{\phi_p} & \prod_{k\in S / (n/p)} B_{kn/p} \ar[d]^{F_p} \\
A_n \ar[r]^-{\gamma_n} & \prod_{k \in S/n} B_{kn}
}
\]
for all prime numbers $p \mid n , n \in S$ are commutative. Here
\[
F_p ((x_k)_{k \in S / (n/p)}) = (x_{pk})_{k \in S/n} \; .
\]
It follows that defining $\phi_m = \phi^{\nu_1}_{p_1} \verk \ldots \verk \phi^{\nu_r}_{p_r}$ for $m = p^{\nu_1}_1 \cdots p^{\nu_r}_r$ as before, the following diagrams for $m \mid n , n \in S$ commute:
\[
\xymatrix{
A_{n/m} \ar[r]^-{\gamma_{n/m}} \ar[d]_{\phi_m} & \prod_{k \in S / (n/m)} B_{kn/m} \ar[d]^{F_m} \\
A_n \ar[r]^-{\gamma_n} & \prod_{k \in S/n} B_{kn} \; .
}
\]
Here $F_m ((x_k)_{k \in S/(n/m)}) = (x_{mk})_{k \in S/n}$. The property $\pi_n \verk \gamma_n = \alpha_n$ implies that for $a \in A_{n/m}$ we have
\[
\alpha_n (\phi_m (a)) = \pi_n (\gamma_n (\phi_m (a))) = \pi_n (F_m (\gamma_{n/m} (a))) = \gamma_{n/m} (a)_m \; .
\]
Setting $\nu = n/m$ it follows that for all $\nu \in S , m \in S/\nu$ and $a \in A_{\nu}$ we have
\[
\gamma_{\nu} (a)_m = \alpha_{m\nu} (\phi_m (a)) \quad \text{in} \; B_{m\nu}
\]
or by replacing the index $m$ by the letter $k$:
\[
\gamma_{\nu} (a) = (\alpha_{k\nu} (\phi_k (a)))_{k \in S/\nu} \quad \text{in} \; \prod_{k \in S/\nu} B_{k\nu} \; .
\]
\end{proof}

\begin{rem}
Consider the map $\beta = (\beta_n) : \ubA \to \ubW (\ubB)$ where
\[
\beta_n : A_n \to \ubW (\ubB)_n = W_{S/n} (F_n (\ubB)) \; .
\]
In the proof above we have seen that the composition of $\beta_n$ with the ghost map
\[
\Gh : W_{S/n} (F_n (\ubB)) \to \prod_{k \in S/n} B_{kn}
\]
is given as follows: For $a \in A_n$ we have the formula:
\[
\Gh (\beta_n (a)) = (\alpha_{kn} (\phi_k (a)))_{k \in S/n} \; .
\]
\end{rem}
\end{small}

\begin{thebibliography}{Aue02}

\bibitem[A]{A}
Roland Auer.
\newblock A functorial property of nested {W}itt vectors.
\newblock {\em J. Algebra}, 252(2):293--299, 2002.

\bibitem[CD]{CD}
Joachim Cuntz and Christopher Deninger.
\newblock Witt vector rings and the relative de {R}ham {W}itt complex.
\newblock {\em J. Algebra}, 440:545--593, 2015.
\newblock With an appendix by Umberto Zannier.

\bibitem[DG]{DG}
Michel Demazure and Pierre Gabriel.
\newblock {\em Groupes alg\'ebriques. {T}ome {I}: {G}\'eom\'etrie alg\'ebrique,
  g\'en\'eralit\'es, groupes commutatifs}.
\newblock Masson \& Cie, \'Editeur, Paris; North-Holland Publishing Co.,
  Amsterdam, 1970.
\newblock Avec un appendice {{\i}t Corps de classes local} par Michiel
  Hazewinkel.

\bibitem[H]{H}
Lars Hesselholt.
\newblock The big de {R}ham-{W}itt complex.
\newblock {\em Acta Math.}, 214(1):135--207, 2015.

\bibitem[L]{L}
Cristian Lenart.
\newblock Formal group-theoretic generalizations of the necklace algebra,
  including a {$q$}-deformation.
\newblock {\em J. Algebra}, 199(2):703--732, 1998.

\bibitem[O1]{O1}
Young-Tak Oh.
\newblock Nested {W}itt vectors and their {$q$}-deformation.
\newblock {\em J.~Algebra}, 309(2):683--710, 2007.

\bibitem[O2]{O2}
Young-Tak Oh.
\newblock Generalizing {W}itt vector construnction.
\newblock ar{X}iv:1211.3508.

\end{thebibliography}

\end{document}